\documentclass{article}

\usepackage{subfig}
\usepackage{amsmath,amsthm}
\usepackage{amsfonts}
\usepackage{bigints}
\usepackage{graphicx}
\usepackage{a4wide}
\usepackage{url}

\newtheorem{theorem}{Theorem}
\newtheorem{proposition}{Proposition}
\newtheorem{remark}{Remark}
\newtheorem{lemma}{Lemma}

\begin{document}
	
\title{Exact solution for a quantum SIR model\thanks{This 
is a preprint of a paper published in 'Journal of Mathematical Sciences'
at https://doi.org/10.1007/s10958-025-08181-6}}

\author{Márcia Lemos-Silva$^1$\\
{\tt marcialemos@ua.pt}
\and
Sandra Vaz$^2$\\ 
{\tt svaz@ubi.pt}
\and 
Delfim F. M. Torres$^1$\thanks{Corresponding author.}\\
{\tt delfim@ua.pt}}

\date{$^1$\text{Center for Research and Development in Mathematics and Applications (CIDMA),}\\ 
Department of Mathematics, University of Aveiro,\\ 
3810-193 Aveiro, Portugal\\[0.3cm]
$^2$Center of Mathematics and Applications (CMA-UBI),\\ 
Department of Mathematics, University of Beira Interior,\\ 
6201-001 Covilh\~{a}, Portugal}

\maketitle


\begin{abstract}
Based on the classical continuous system initially proposed by Bailey in 
1975, we present a novel Susceptible--Infected--Removed (SIR) model defined in quantum 
time, where the temporal evolution is governed by a non-uniform time grid. An explicit 
analytical solution is derived, and the long-term behavior of the susceptible, infected, 
and removed individuals is analyzed. Moreover, we prove the model preserves dynamic 
consistency with its continuous counterpart, as evidenced by the non-negativity of 
solutions and their corresponding qualitative agreement with the continuous dynamics. 
All results are further supported by illustrative examples.

\medskip


\end{abstract}


\section{Introduction}

The Susceptible--Infected--Removed (SIR) model is well-established 
and has been extensively studied over the years \cite{MR4795685}. 
Despite its long history, it continues to be of current great interest, 
since it provides a solid foundation for understanding and predicting 
the spread of diseases within populations 
\cite{MR4755935,MR4750629,MR4754098}.

In \cite{bailey}, Norman Bailey proposed the following SIR model:
\begin{equation}
\label{continuous}
\begin{cases}
x' = -\dfrac{b x y}{x+y}, \\
y' = \dfrac{b x y}{x+y} - cy, \\
z' = cy,
\end{cases}
\end{equation}
where $b, c \in \mathbb{R}^+$, and where $x(t_0) = x_0$, $y(t_0) = y_0$, and $z(t_0) = z_0$ 
are non-negative and, moreover, $x, y, z : \mathbb{R} \rightarrow \mathbb{R}^+_0$. This model 
adopts a well-known compartmental approach, wherein three variables represent different groups 
within the population. The variable $x$ represents the proportion of susceptible individuals, 
$y$ represents the proportion of infected individuals, 
and $z$ represents the proportion of removed individuals. 
By adding the right-hand side of equations \eqref{continuous}, one has $x'(t) + y'(t) + z'(t) = 0$, 
which is equivalent to $x(t) + y(t) + z(t) = N$, for all $t$, where the constant $N:= x_0 + y_0 + z_0$  
denotes the total population under study. The exact solution of \eqref{continuous} 
was found in \cite{gleissner:solution}. More recently, a different method was proposed 
to solve \eqref{continuous}, generalizing the result of \cite{gleissner:solution} 
to non-autonomous cases \cite{scales:solution}.

Quantum calculus, also known as $q-$calculus, is a branch of mathematics that extends 
the traditional calculus to a framework that incorporates $q-$analogues, that is, 
mathematical structures that replace certain variables with $q-$deformations, 
where $q$ is a parameter that controls the deformation \cite{quantum,quantumVC}. 
In the context of quantum systems 
or quantum phenomena, mathematical modeling often requires specialized mathematical frameworks 
that can account for unique properties of quantum mechanics. Mathematical models built within 
quantum calculus can be used to study a wide range of phenomena, such as quantum entanglement, 
quantum information, quantum statistical mechanics, and quantum field theory. These models aid 
in understanding the behavior of quantum systems, predicting their properties.

Here, we propose and investigate a SIR model formulated in quantum time. 
One of the main motivations for considering such a model lies in the non-uniform 
structure of the quantum time scale. Unlike the standard discrete-time framework, 
where time progresses in uniform steps, quantum time evolves on a increasing step. 
This means that during the early stages of an epidemic, where dynamics tend 
to be more sensitive, the step is finer, turning to a coarser one as the system 
stabilizes. Quantum modeling offers a more natural 
discretization for many real-world phenomena, where the initial dynamics 
changes rapidly before approaching an equilibrium. To the best of our knowledge, 
no SIR model has yet been formulated in quantum time, making our approach 
a novel contribution to the field. Thus, regarding applications, 
the quantization of infectious disease modeling may
effectively address some limitations of traditional continuous 
compartmental models by introducing a novel approach to understand disease dynamics.
Furthermore, unlike many SIR models, our proposed model has an exact solution. 
Finally, we prove that all its solutions remain non-negative and that their long-term behavior 
are coherent with \eqref{continuous}, which guarantees consistency.

To describe the dynamics of the model, we make use of the $q$-derivatives 
defined in quantum calculus as
\begin{equation}
\label{derivative}
D_qf(x) = \frac{f(qx) - f(x)}{(q-1)x},
\end{equation} 
where $D_qf$ denotes the $q$-derivative of function $f$. Moreover, $f(x)$ and $f(qx)$ represent 
the value of the function at $x$ and $qx$, respectively, and $(q-1)x$ denotes the step size that, as usual, 
measures the interval between $x$ and $qx$. Note that, when $q$ approaches 1, 
the $q$-derivative recovers its classical counterpart. 

The paper is organized as follows.  In Section~\ref{sec2}, we introduce the SIR model using 
$q-$derivatives and derive the explicit $q-$model. Since our goal is to prove dynamic consistency, 
we begin by proving the non-negativity and boundedness of solutions. 
In the same section, we find the equilibrium points and derive the exact solution. 
Section~\ref{sec2} ends with the asymptotic behavior of the solution. 
We end our work with Section~\ref{sec3}, providing some examples 
that illustrate the obtained results. 


\section{Main Results}
\label{sec2}

To obtain the quantum system, we extend the Mickens method \cite{mickens:book}
to the framework of quantum calculus \cite{quantum,quantumVC}. 
Such method is a Nonstandard Finite Difference (NSFD) scheme that, unlike classical methods, such as Euler 
and Runge--Kutta, ensures dynamic consistency \cite{stuart:consistent}. Developed by Mickens, this method 
presents some rules specifically designed to avoid the instabilities and inconsistencies common 
in standard discretization methods \cite{mickens:book,mickens,mickens1}. This method has already 
been successfully applied to various systems, e.g., \cite{applied:mickens,mickens:applied,vaz:torres,vaz:torres:sica}. 
However, in our work, since we are using the concept of $q$-derivatives for the approximation of derivatives, 
we follow only one of Mickens' rules, which states that the linear and nonlinear terms of the system may need 
to be replaced by nonlocal terms. As we shall see, this rule will be crucial for proving the non-negativity 
and boundedness of solutions.

Consider the quantum SIR model of the form 
\begin{equation}
\label{quantum}
\begin{cases}
D_qx(t) = -\dfrac{bx(qt)y(t)}{x(t) + y(t)},\\
D_qy(t) = \dfrac{bx(qt)y(t)}{x(t) + y(t)} - cy(qt),\\
D_qz(t) = cy(qt),
\end{cases}
\end{equation}
where $D_qx(t)$, $D_qy(t)$ and $D_qz(t)$ represent the $q$-derivatives with respect to $x$, $y$ and $z$, 
respectively, and $qt$ denotes the moment following $t$. Moreover, $x(t_0) > 0$, $y(t_0) > 0$ 
and $z(t_0) \geq 0$, and $x, y : \mathbb{R}^+ \rightarrow \mathbb{R}^+_0$, 
$z: \mathbb{R}^+_0 \rightarrow \mathbb{R}^+_0$, $b, c \in \mathbb{R}^+$ and $q > 1$.
To compare our model \eqref{quantum} with the ones available in the literature
one can take the limit when $q$ tends to one from the right. 
In the limit, when $q \rightarrow 1^+$, model \eqref{quantum} converges to the
classical continuous SIR model \eqref{continuous}.

Substituting \eqref{derivative} in \eqref{quantum}, we obtain \eqref{quantum} in the form
\begin{equation}
\label{quantum:explicit}
\begin{cases}
x(qt) = \dfrac{x(t)(x(t) + y(t))}{x(t) + y(t)(1 + b(q-1)t)},\\ 
y(qt) = \dfrac{y(t)(1 + b(q-1)t)(x(t) + y(t))}{(1 + c(q-1)t)(x(t) + y(t)(1 + b(q-1)t))},\\
z(qt) = \dfrac{c(q-1)t y(t)(1 + b(q-1)t)(x(t) + y(t))}{(1 + c(q-1)t)(x(t) + y(t)(1 + b(q-1)t))} + z(t).
\end{cases}
\end{equation}


\subsection{Non-negativity and boundedness of solutions}
\label{subsec1}

Ensuring the non-negativity of solutions and maintaining their boundedness 
are both crucial properties in epidemiological models. 
Follows our first result.

\begin{proposition}
If the initial conditions are non-negative and the parameter values 
of the system \eqref{quantum:explicit} are positive, then all the 
solutions remain positive for all $t >t_0$. Moreover, the quantum model 
\eqref{quantum} guarantees that the population remains constant over time.
\end{proposition}

\begin{proof}
Let $(x(t_0), y(t_0), z(t_0)) \in \mathbb{R}^3_+$. Since all parameters are positive 
and, by definition, $q>1$, then all equations of system \eqref{quantum:explicit} 
are positive. Thus, the non-negativity condition is trivially satisfied.
Now, let the total population be defined as $N(t) = x(t) + y(t) + z(t)$. 
Adding the right-hand side of equations \eqref{quantum}, we have $D_qN(t) = 0$, 
which means, by the properties of the quantum derivative \cite{quantum,quantumVC}, 
that the population $N(t)$ remains constant along time $t$. 
\end{proof}


\subsection{Equilibrium points}
\label{subsec2}

To have dynamical consistency, the equilibrium points of the quantum model 
have to be the same of the continuous model \eqref{continuous}.

Equating all the $q$-derivatives to zero, it follows easily that the equilibrium 
points of system \eqref{quantum} are given by $(\alpha, 0, N - \alpha)$, 
with $\alpha \geq 0$. This is coherent with the equilibrium points 
of the continuous-time model \eqref{continuous} stated in \cite{bailey}.


\subsection{Exact solution}
\label{subsec3}

Here, we derive the exact solution of system \eqref{quantum:explicit}. 
Note that, since $N = x(qt) + y(qt) + z(qt)$ for all $t$, it means it is 
enough to solve the first two equations of system \eqref{quantum:explicit}: 
after we know $x(t)$ and $y(t)$, we can immediately compute $z(qt)$ 
using the equality $z(qt) = N - x(qt) - y(qt)$.

\begin{remark}
In the following results, we use the notation $x(t_n)$, 
where $n \in \mathbb{N}_0$. By definition, 
$t_1 = qt_0$, $t_2 = qt_1 = q^2t_0$, \ldots, i.e., $t_n = q^nt_0$. 
\end{remark}

\begin{remark}
Let $x_0,x_1,x_2,\dots$ be a sequence. Then, by convention, the product 
$\prod_{i = n}^m x_i$ is equal to 1 when $n > m$. 
\end{remark}

\begin{theorem}
The exact solution of system \eqref{quantum:explicit} is given by
\begin{equation}
\label{solution}
\begin{cases}
x(t_n) = \displaystyle \frac{x(t_0)\left(\bar{\kappa} + 1\right)}{\bar{\kappa}
+1+b(q-1)t_0} \displaystyle\prod_{i=0}^{n-2} \frac{1 + \bar{\kappa}
\prod_{j=0}^{i} \xi_j}{1 + b(q-1)q^{i+1}t_0 + \bar{\kappa}\prod_{j=0}^{i} \xi_j},\\ \\
y(t_n) = \frac{y(t_0)}{\xi_0}\frac{\bar{\kappa} + 1}{\bar{\kappa}
+1+b(q-1)t_0}\displaystyle\prod_{i=0}^{n-2} \left(\frac{1}{\xi_{i+1}}
\cdot\frac{1 + \bar{\kappa}\prod_{j=0}^{i} \xi_j}{1 + b(q-1)q^{i+1}t_0 
+ \bar{\kappa}\prod_{j=0}^{i} \xi_j}\right),\\ \\
z(t_n) = N - \left(x(t_0) + \frac{y(t_0)}{\xi_0}\displaystyle\prod_{i=0}^{n-2} 
\left(\frac{1}{\xi_{i+1}}\right)\right)\frac{\bar{\kappa} + 1}{\bar{\kappa}
+1+b(q-1)t_0} \displaystyle\prod_{i=0}^{n-2} \frac{1 + \bar{\kappa}
\prod_{j=0}^{i} \xi_j}{1 + b(q-1)q^{i+1}t_0 + \bar{\kappa}\prod_{j=0}^{i} \xi_j},
\end{cases}
\end{equation}
where $\bar{\kappa} = \frac{x(t_0)}{y(t_0)}$ 
and $\xi_k = \frac{1 + c(q-1)q^kt_0}{1 + b(q-1)q^kt_0}$, 
$k \in \{0,1,\dots, n-1\}$.
\end{theorem}

\begin{proof}
We do the proof by induction. For $n = 1$, we have
\begin{equation}
\begin{aligned}
\label{induction:x1}
x(t_1) &= x(t_0)\cdot\frac{\bar{\kappa}+1}{\bar{\kappa}+1+b(q-1)t_0} \\
&= \frac{x(t_0)(x(t_0)+y(t_0))}{x(t_0)+y(t_0)(1+b(q-1)t_0)},
\quad \text{by the definition of}\, \bar{\kappa},	
\end{aligned}
\end{equation}	
and
\begin{equation}
\begin{aligned}
\label{induction:y1}
y(t_1) &= \frac{y(t_0)}{\xi_0}\cdot\frac{\bar{\kappa}+1}{\bar{\kappa}+1+b(q-1)t_0} \\
&= \frac{y(t_0)(1+b(q-1)t_0)(x(t_0)+y(t_0))}{(1+c(q-1)t_0)(x(t_0)+y(t_0)(1+b(q-1)t_0))},
\quad \text{by the definition of}\, \bar{\kappa} \,\text{and}\,~ \xi_0.
\end{aligned}
\end{equation}	
Both \eqref{induction:x1} and \eqref{induction:y1} are true from system
\eqref{quantum:explicit}. Now, let us state the inductive hypothesis 
that \eqref{solution} holds true for a certain $n = m$. 
We want to prove that it remains valid for $n = m+1$. Thus, 
\begin{equation*}
\begin{aligned}
x(t_{m+1}) &= x(t_0) \cdot\frac{\bar{\kappa}+1}{\bar{\kappa}+1+b(q-1)t_0}
\displaystyle \prod_{i=0}^{m-1} \frac{1 + \bar{\kappa}\prod_{j=0}^{i}\xi_j}{1 
+ b(q-1)q^{i+1}t_0 + \bar{\kappa}\prod_{j=0}^{i}\xi_j}\\
&= x(t_0) \cdot \frac{\bar{\kappa}+1}{\bar{\kappa}+1+b(q-1)t_0}\displaystyle \prod_{i=0}^{m-2} 
\frac{1 + \bar{\kappa}\prod_{j=0}^{i}\xi_j}{1 + b(q-1)q^{i+1}t_0 + \bar{\kappa}\prod_{j=0}^{i}\xi_j} \\
&\cdot \frac{1 + \bar{\kappa}\prod_{j=0}^{m-1}\xi_j}{1 + b(q-1)q^{m}t_0 + \bar{\kappa}\prod_{j=0}^{m-1}\xi_j} \\
&= x(t_m) \cdot \frac{1 + \bar{\kappa}\prod_{j=0}^{m-1}\xi_j}{1 + b(q-1)q^{m}t_0 
+ \bar{\kappa}\prod_{j=0}^{m-1}\xi_j},\,\text{by inductive hypothesis}.
\end{aligned}
\end{equation*}
Now, note that the first two equations of \eqref{solution}, 
that hold true by the inductive hypothesis, can be rewritten as
\begin{equation}
\label{induction:x0}
x(t_0) = \frac{x(t_n)}{\frac{\bar{\kappa} + 1}{\bar{\kappa} + 1 + b(q-1)t_0}
\displaystyle \prod_{i=0}^{n-2} \frac{1 + \bar{\kappa}
\prod_{j=0}^{i}\xi_j}{1 + b(q-1)q^{i+1}t_0 + \bar{\kappa}\prod_{j=0}^{i}\xi_j}}
\end{equation}
and
\begin{equation}
\label{induction:y0}
y(t_0) = \frac{y(t_n)\xi_0}{\frac{\bar{\kappa} + 1}{\bar{\kappa} + 1 + b(q-1)t_0}
\displaystyle \prod_{i=0}^{n-2} \left(\frac{1}{\xi_{i+1}}
\cdot\frac{1 + \bar{\kappa}\prod_{j=0}^i \xi_j}{1 + b(q-1)q^{i+1}t_0 
+ \bar{\kappa}\prod_{j=0}^i\xi_j}\right)}.
\end{equation}
Using \eqref{induction:x0} and \eqref{induction:y0}, 
and the definition of $\bar{\kappa}$, one can write that
\begin{equation}
\label{aux:k}
\begin{aligned}
\bar{\kappa} \prod_{j=0}^{m-1}\xi_j 
&= \frac{x(t_m) \prod_{i=0}^{m-2}\frac{1}{\xi_{i+1}}}{y(t_m)\xi_0} \cdot\prod_{j=0}^{m-1}\xi_j \\
&= \frac{x(t_m)}{y(t_m)\prod_{i=0}^{m-1} \xi_i}\cdot\prod_{j=0}^{m-1}\xi_j \\
&= \frac{x(t_m)}{y(t_m)}.  
\end{aligned}
\end{equation}
Thus, applying \eqref{aux:k}, $x(t_{m+1})$ becomes
\begin{equation}
\label{induction:xm1}	
\begin{aligned}
x(t_{m+1}) &= x(t_m) \cdot \frac{x(t_m) + y(t_m)}{x(t_m) + y(t_m)(1+b(q-1)q^{m}t_0)}\\
&= x(t_m) \cdot \frac{x(t_m) + y(t_m)}{x(t_m) + y(t_m)(1+b(q-1)t_m)}, 
\quad\text{because}\, q^mt_0 = t_m. 
\end{aligned}
\end{equation}
Moreover, 
\begin{equation*}
\begin{aligned}
y(t_{m+1}) &= y(t_0)  \cdot\frac{\bar{\kappa}+1}{\bar{\kappa}+1+b(q-1)t_0}
\displaystyle \prod_{i=0}^{m-1}\left( \frac{1}{\xi_{i+1}}
\cdot\frac{1 + \bar{\kappa}\prod_{j=0}^{i}\xi_j}{1 + b(q-1)q^{i+1}t_0 
+ \bar{\kappa}\prod_{j=0}^{i}\xi_j}\right)\\
&= y(t_0) \cdot \frac{\bar{\kappa}+1}{\bar{\kappa}+1+b(q-1)t_0}
\displaystyle \prod_{i=0}^{m-2} \left(\frac{1}{\xi_{i+1}}
\cdot\frac{1 + \bar{\kappa}\prod_{j=0}^{i}\xi_j}{1 + b(q-1)q^{i+1}t_0 
+ \bar{\kappa}\prod_{j=0}^{i}\xi_j}\right)\\
&\cdot \frac{1}{\xi_m}\cdot\frac{1 + \bar{\kappa}\prod_{j=0}^{m-1}\xi_j}{1 
+ b(q-1)q^{m}t_0 + \bar{\kappa}\prod_{j=0}^{m-1}\xi_j} \\
&= \frac{y(t_m)}{\xi_m} \cdot \frac{1 + \bar{\kappa}\prod_{j=0}^{m-1}\xi_j}{1 
+ b(q-1)q^{m}t_0 + \bar{\kappa}\prod_{j=0}^{m-1}\xi_j},\,
\text{by inductive hypothesis}.
\end{aligned}
\end{equation*}
Finally, given that $q^mt_0 = t_m$, and applying \eqref{aux:k} and 
the definition of $\xi_m$, it follows that
\begin{equation}
\label{induction:ym1}
y(t_{m+1}) = \frac{y(t_m)(1 + b(q-1)t_m)(x(t_m)+y(t_m))}{(1
+c(q-1)t_m)(x(t_m)+y(t_m)(1+b(q-1)t_m))}.
\end{equation}
We end by noting that both \eqref{induction:xm1} and \eqref{induction:ym1} are 
also in agreement with system \eqref{quantum:explicit}. 
Furthermore, since $N = x(t_n) + y(t_n) + z(t_n)$, we have
\begin{equation*}
z(t_n) = N - \left(x(t_0) + \frac{y(t_0)}{\xi_0}\displaystyle\prod_{i=0}^{n-2} 
\left(\frac{1}{\xi_{i+1}}\right)\right)
\frac{\bar{\kappa} + 1}{\bar{\kappa}+1+b(q-1)t_0}
\displaystyle\prod_{i=0}^{n-2} \frac{1 + \bar{\kappa}\prod_{j=0}^{i} \xi_j}{1 + b(q-1)q^{i+1}t_0 
+ \bar{\kappa}\prod_{j=0}^{i} \xi_j}.
\end{equation*}
The proof is complete. 
\end{proof}

Note that $x(t_n)$ and $y(t_n)$ of \eqref{solution} can be rewritten as
\begin{equation}
\label{xtn_simplified}
x(t_n) = x(t_0)\frac{\bar{\kappa} + 1}{\bar{\kappa} + 1 + b(q-1)t_0} 
\displaystyle \prod_{i=0}^{n-2} \frac{1}{1 + \frac{b(q-1)q^{i+1}t_0}{1 
+ \bar{\kappa}\prod_{j=0}^{i} \xi_j}},
\end{equation}
and 
\begin{equation}
\label{ytn_simplified}
y(t_n) = \frac{y(t_0)}{\xi_0} \cdot \frac{\bar{\kappa} + 1}{\bar{\kappa} 
+ 1 + b(q-1)t_0} \displaystyle \prod_{i=0}^{n-2} \left(\frac{1}{\xi_{i+1}} 
\cdot \frac{1}{1 + \frac{b(q-1)q^{i+1}t_0}{1 + \bar{\kappa}\prod_{j=0}^{i} \xi_j}}\right).
\end{equation}

The reproduction number $\mathcal{R}_0$ is one of the most important concepts 
in epidemiology. For system \eqref{continuous}, it is well-known that 
\begin{equation*}
\mathcal{R}_0 = \frac{b}{c}.
\end{equation*}
Since the underlying principles of disease transmission and recovery do not change, 
the expression for $\mathcal{R}_0$ remains valid for the quantum time system
\eqref{quantum} or \eqref{quantum:explicit}. In the following results, 
the focus is on the asymptotic stability of the equilibrium points 
considering the value of $\mathcal{R}_0$. 

 \begin{theorem}
\label{convergence_zero}
Assume that
\begin{equation}
\label{condition}
\frac{1}{\xi_{i+1}} \cdot 
\frac{1}{1 + \frac{b(q-1)q^{i+1}t_0}{1 + \bar{\kappa}\prod_{j=0}^{i} \xi_j}} < 1, 
\quad \text{for all} \\\ i \geq 0.
\end{equation}
If $\mathcal{R}_0 \geq 1$, then all solutions 
\eqref{solution} of system \eqref{quantum:explicit}
converge to the equilibrium point $(0,0,N)$. 
\end{theorem}  

\begin{proof}
First, let us consider $\mathcal{R}_0 = 1$, that is, $b = c$. Then, 
\begin{equation*}
\xi_j = \frac{1 + c(q-1)q^jt_0}{1 + b(q-1)q^jt_0} = 1.
\end{equation*}
This means that 
\begin{equation*}
x(t_n) = x(t_0) \frac{\bar{\kappa} + 1}{\bar{\kappa} + 1 + b(q-1)t_0}  
\displaystyle \prod_{i=0}^{n-2} \frac{1}{1 + \frac{b(q-1)q^{i+1}t_0}{1 + \bar{\kappa}}}.
\end{equation*}
Clearly, $\frac{\bar{\kappa} + 1}{\bar{\kappa} + 1 + b(q-1)t_0} < 1$ 
and $\frac{1}{1 + \frac{b(q-1)q^{i+1}t_0}{1 + \bar{\kappa}}} < 1$ and, this way, 
one can easily conclude that $\underset{n \to \infty}{\lim}x_{n} =0$. The same 
inference can be drawn regarding $y(t_n)$. Moreover, since $z(t) = N - x(t) - y(t)$, 
it is clear that under these circumstances all solutions 
of \eqref{solution} converge to $(0,0,N)$.

Let us now consider the remaining case where $\mathcal{R}_0 > 1$, that is, $b > c$. 
This leads to $\xi_j  < 1$. Considering \eqref{xtn_simplified}, we can state $a_i$ as
\begin{equation}
\label{a_xtn}
a_i = \frac{1}{1 + \frac{b(q-1)q^{i+1}t_0}{1 + \bar{\kappa}\prod_{j=0}^{i} \xi_j}}, 
\quad i \geq 0.
\end{equation}
So, since $\xi_j < 1$ and $q > 1$, we have
\begin{equation}
\label{condition_ai}
1 > \frac{1}{1 + \frac{b(q-1)q^{i+1}t_0}{1 + \bar{\kappa}\prod_{j=0}^{i} \xi_j}} 
> \frac{1}{1 + \frac{b(q-1)q^{i+2}t_0}{1 + \bar{\kappa}\prod_{j=0}^{i+1} \xi_j}} 
\Rightarrow 1 > a_i > a_{i+1}, 
\quad i \geq 0.
\end{equation}
Furthermore, as $\frac{\bar{\kappa} + 1}{\bar{\kappa} + 1 + b(q-1)t_0} < 1$, 
we have
\begin{align*}
0 < x(t_n) &= x(t_0) \frac{\bar{\kappa} + 1}{\bar{\kappa} + 1 
+ b(q-1)t_0} a_0 a_1 a_2 \dots a_{n-2} \\
&< x(t_0) \frac{\bar{\kappa} + 1}{\bar{\kappa} + 1 + b(q-1)t_0}(a_0)^{n-1},
\end{align*}
leading to $\underset{n \to \infty}{\lim}x(t_n) =0$. 

Considering now \eqref{ytn_simplified}, we can define $\tilde{a}_i$ as
\begin{equation}
\label{a_ytn}
\tilde{a}_i = \frac{1}{\xi_{i+1}} \cdot 
\frac{1}{1 + \frac{b(q-1)q^{i+1}t_0}{1 + \bar{\kappa}\prod_{j=0}^{i} \xi_j}}, 
\quad i \geq 0.
\end{equation}
From the previous case, we already know that
\begin{equation*}
\frac{1}{1 + \frac{b(q-1)q^{i+1}t_0}{1 + \bar{\kappa}
\prod_{j=0}^{i} \xi_j}}, \quad i \geq 0,
\end{equation*}
tends to zero as $i$ grows.
Now we prove that $\xi_{i+1} < \xi_i < 1$ for all $i \geq 0$. 
For a certain $i\geq 0$, suppose that $\xi_{i+1} < \xi_i$. Then, 
\begin{equation*}
\frac{1 + c(q-1)q^{i+1}t_0}{1 + b(q-1)q^{i+1}t_0} 
< \frac{1 + c(q-1)q^{i}t_0}{1 + b(q-1)q^{i}t_0}, 
\end{equation*}
which is equivalent to
\begin{align*}
&b(q-1)q^it_0 - b(q-1)q^{i+1}t_0 + c(q-1)q^{i+1}t_0 - c(q-1)q^it_0 < 0 \\ 
&\Leftrightarrow (c-b)(q-1)^2q^it_0 < 0.
\end{align*}
This is always true since, by definition, $b > c$ and $q > 1$. 
This means that $\xi_{i+1} < \xi_i$ for all $i \geq 0$.
Combined with \eqref{condition} and \eqref{condition_ai}, this leads to 
\begin{equation*}
\frac{1}{\xi_{i+1}}\cdot\frac{1}{1 + \frac{b(q-1)q^{i+1}t_0}{1 
+ \bar{\kappa}\prod_{j=0}^{i} \xi_j}} < \frac{1}{\xi_{i+2}}
\cdot\frac{1}{1 + \frac{b(q-1)q^{i+2}t_0}{1 + \bar{\kappa}
\prod_{j=0}^{i+1} \xi_j}} < 1 \Rightarrow \tilde{a}_i < \tilde{a}_{i+1} < 1, 
\quad i \geq 0.
\end{equation*}
Therefore,
\begin{align*}
0 < y(t_n) &= \frac{y(t_0)}{\xi_0}\frac{\bar{\kappa} + 1}{\bar{\kappa} 
+ 1 + b(q-1)t_0} \tilde{a}_0 \tilde{a}_1 \tilde{a}_2 \dots \tilde{a}_{n-2}\\
&< \frac{y(t_0)}{\xi_0} \frac{\bar{\kappa} + 1}{\bar{\kappa} + 1 + b(q-1)t_0}(\tilde{a}_{n-2})^{n-1},
\end{align*}
leading to $\underset{n \to \infty}{\lim}y(t_n) =0$. Also, 
$\underset{n \to \infty}{\lim}z(t_n) = N$. Thus, all solutions of \eqref{solution} 
converge to the equilibrium point $(0,0,N)$ under the condition $b>c$.
\end{proof}	

To address the case when $\mathcal{R}_0 < 1$, we make use of the following lemma.

\begin{lemma}
\label{lemma}
Let $\mathcal{R}_0 < 1$, that is, $b < c$. Then, $\xi_j > 1$ for all $j\geq 0$. 
Moreover, considering \eqref{a_xtn}, then $a_{i} > a_{i+1}$ for all $i\geq 0$.
\end{lemma}

\begin{proof}
Let us consider that $a_{i} > a_{i+1}, \forall i\geq 0$. This means that 
\begin{equation*}
\frac{b(q-1)q^{i+1}t_0}{1 + \bar{\kappa}\prod_{j=0}^{i}\xi_{j}} 
< \frac{b(q-1)q^{i+2}t_0}{1 + \bar{\kappa}\prod_{j=0}^{i+1}\xi_{j}},
\end{equation*}
which is equivalent to 
\begin{equation*}
q^{i+1}\left(1 + \bar{\kappa}\displaystyle \prod_{j=0}^{i+1}\xi_j\right) 
< q^{i+2}\left(1 + \bar{\kappa}\displaystyle \prod_{j=0}^{i}\xi_j\right)
\Leftrightarrow (1-q) + \bar{\kappa}\prod_{j=0}^{i}\xi_j(\xi_{i+1} - q) < 0.
\end{equation*}
Since $q > 1$, the previous relation is always true if $\xi_{i+1} - q < 0$, that is,
\begin{equation*}
\frac{1+c(q-1)q^{i+1}t_0}{1+b(q-1)q^{i+1}t_0} < q. 
\end{equation*}
Simplifying, we obtain that
\begin{equation}
\label{condition_q}
1 + (c-bq)(q-1)q^{i+1}t_0 - q < 0,
\end{equation}
which is always true provided 
$$
c - bq < 0 \Leftrightarrow q > \frac{c}{b}.
$$ 
On the other hand, \eqref{condition_q} is equivalent to 
\begin{equation}
\label{condition_2q}
(1-q)(1-(c-bq)q^{i+1}t_0) < 0.
\end{equation}
Since $q > 1$, \eqref{condition_2q} holds true if and only if 
$$
(1-(c-bq)q^{i+1}t_0) > 0, 
$$
which is equivalent to 
\begin{equation}
\label{condition_3q}
\frac{1}{(c-bq)q^{i+1}} > t_0 > 0.
\end{equation}
Finally, \eqref{condition_3q} is true if and only if
$$
c - bq > 0 \Leftrightarrow q < \frac{c}{b}.
$$
If $q = \frac{c}{b}$, from \eqref{condition_2q} we get
$$
1 - \frac{c}{b} < 0.
$$
Since, by hypothesis, $\mathcal{R}_0 < 1$, then $\frac{c}{b} > 1$,
and thus, the result holds for any $q > 1$. 
\end{proof}

\begin{theorem}
\label{thm:5}
Let $\mathcal{R}_0 < 1$. Then, under the conditions
of Lemma~\ref{lemma}, all solutions \eqref{solution} of system \eqref{quantum:explicit}
converge to the equilibrium $(\alpha,0,N-\alpha)$ for some $\alpha \in \, ]0,N]$. 
\end{theorem} 

\begin{proof}
Since $\mathcal{R}_0 < 1$, that is, $b < c$, then $\xi_j > 1$. Note that considering 
\eqref{a_xtn}, it is clear that $a_i < 1, \forall i \geq 0$. From Lemma~\ref{lemma}, 
we have $1 > a_i > a_{i+1}$, for any $q > 1$. So,
\begin{align*}
x(t_n) &= x(t_0) \frac{\bar{\kappa} + 1}{\bar{\kappa} + 1 + b(q-1)t_0} 
\displaystyle \prod_{i=0}^{n-2} a_i > x(t_0) \frac{\bar{\kappa} + 1}{\bar{\kappa} 
+ 1 + b(q-1)t_0} \displaystyle \prod_{i=0}^{n-2} a_{n-2} \\ \Rightarrow x(t_n) 
&>  x(t_0) \frac{\bar{\kappa} + 1}{\bar{\kappa} + 1 + b(q-1)t_0} (a_{n-2})^{n-1}. 
\end{align*}
Since $(a_{n-2})^{n-1} \longrightarrow 0$ as $n\to \infty$, 
then $x(t_n) = x(t_0) \frac{\bar{\kappa} + 1}{\bar{\kappa} + 1 + b(q-1)t_0} 
\displaystyle \prod_{i=0}^{n-2} a_i = \alpha$, meaning that 
$\underset{n \to \infty}{\lim}x(t_n) = \alpha$, with $\alpha \in \, ]0,N]$.
Having this, and recalling $\tilde{a}_i$ from \eqref{a_ytn}, 
it just remains left to verify that $\prod_{i=0}^{n-2} \frac{1}{\xi_{i+1}} \longrightarrow 0$. 
This is clearly true since in this case $\xi_{i+1} > 1$ for all $i\geq0$. Therefore, 
$\underset{n \to \infty}{\lim}y(t_n) = 0$. It is now trivial to conclude that 
$\underset{n \to \infty}{\lim}z(t_n) = N - \alpha$. From what we have shown, 
all solutions \eqref{solution} converge to the equilibrium $(\alpha,0,N-\alpha)$. 
\end{proof}


\section{Examples}
\label{sec3}

In this section, we present some illustrative examples confronting systems 
\eqref{continuous} and \eqref{quantum} under various parameter values. 
We start by plotting in Figure~\ref{Fig.1} the case where $b > c$, 
in which can be observed that the solution converges to the equilibrium point $(0,0,N)$ 
in both continuous and quantum cases. 

\begin{figure}[ht!]
\centering
\subfloat[Solution of \eqref{continuous}]{\includegraphics[scale=0.35]{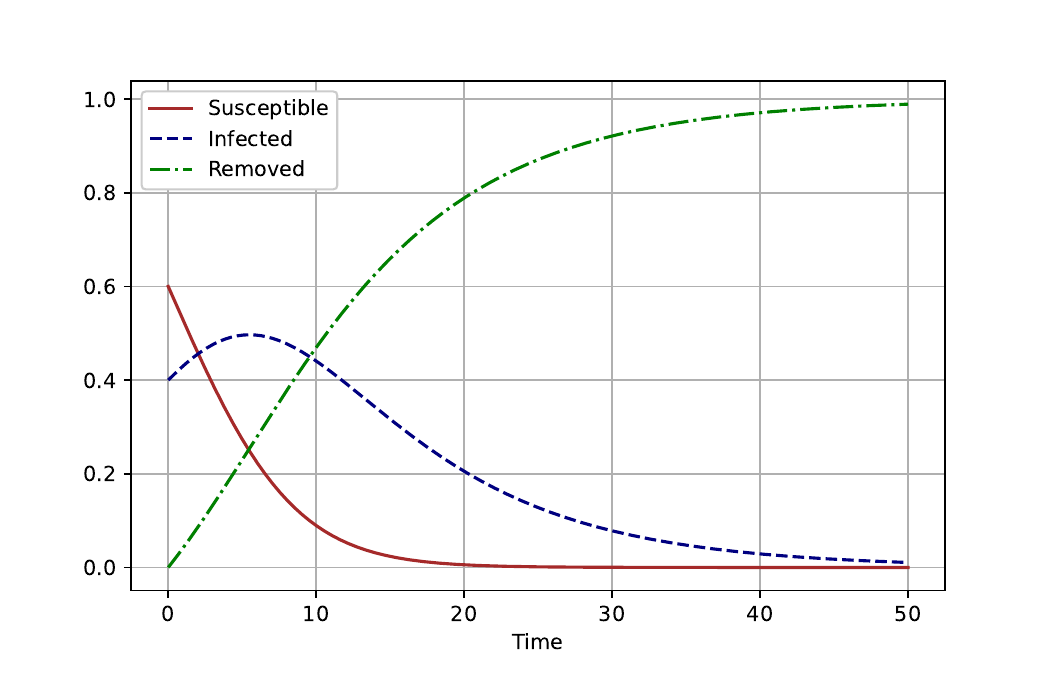}}
\subfloat[Solution of \eqref{quantum:explicit}]{\includegraphics[scale=0.35]{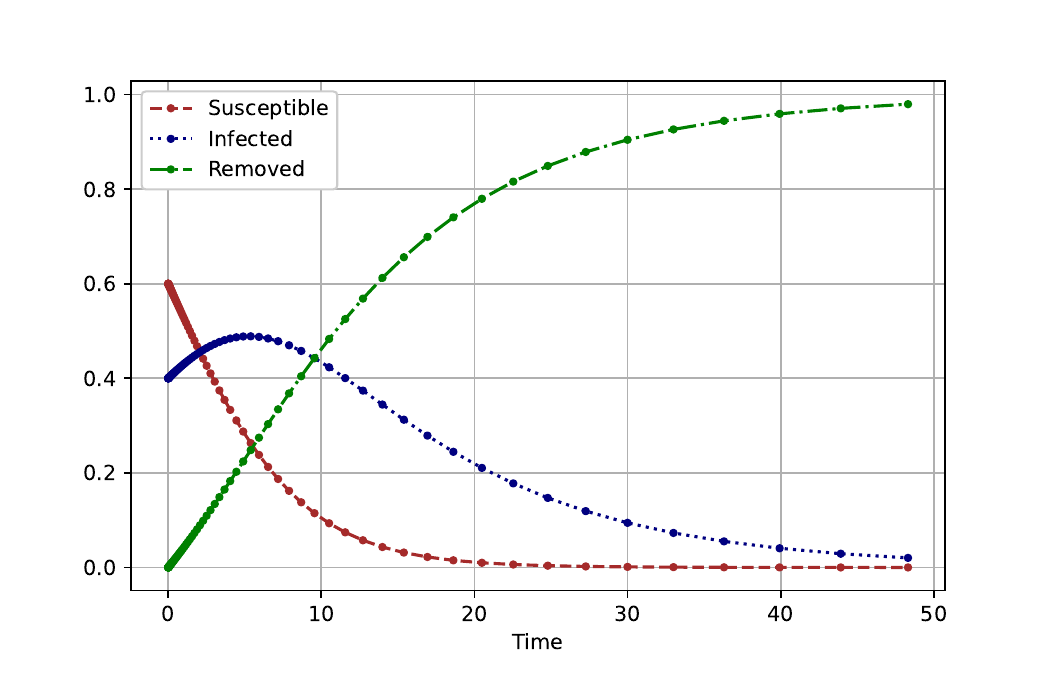}}
\caption{Susceptible, Infected and Removed individuals with $b = 0.3$, $c = 0.1$, 
$x(t_0) = 0.6$, $y(t_0) = 0.4$, $z(t_0) = 0$, $q = 1.1$, and $t_0 = 0.01$.}
\label{Fig.1}
\end{figure}

In contrast, Figure~\ref{Fig.2} examines the scenario where the parameters satisfy 
$b < c$. In this case, it is observable that the solution converges to the equilibrium 
($\alpha, 0, N - \alpha$), where $\alpha \in \mathbb{R}^+$. These examples align with 
the results obtained in Section~\ref{sec2}, reinforcing consistency between 
continuous and quantum SIR models. 

\begin{figure}[ht!]
\centering
\subfloat[Solution of \eqref{continuous}]{\includegraphics[scale=0.35]{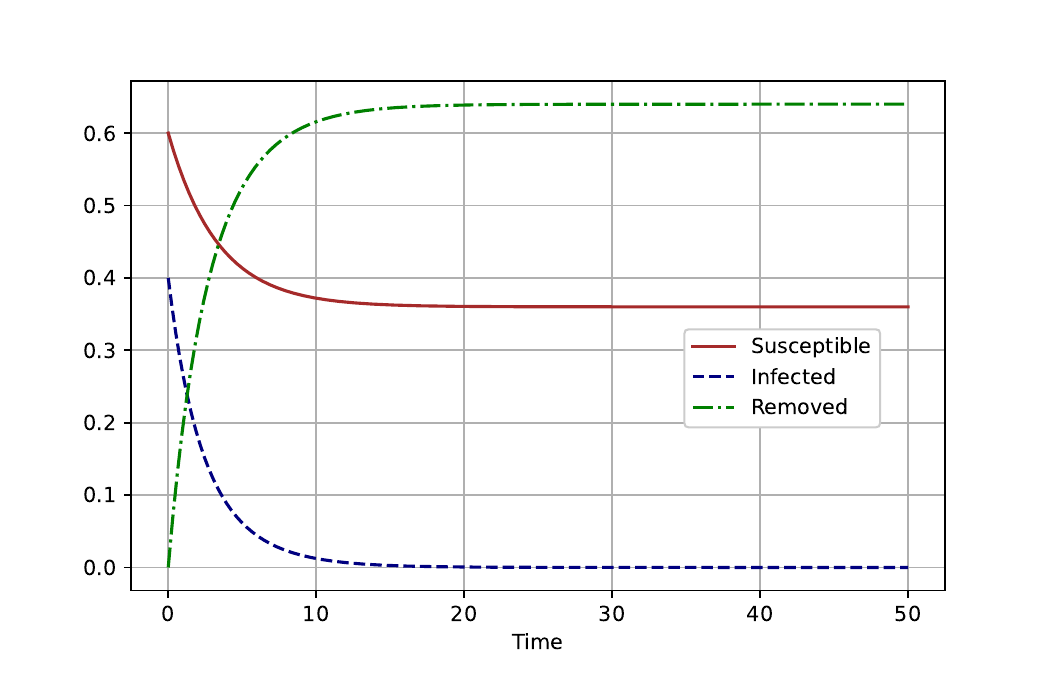}}
\subfloat[Solution of \eqref{quantum:explicit}]{\includegraphics[scale=0.35]{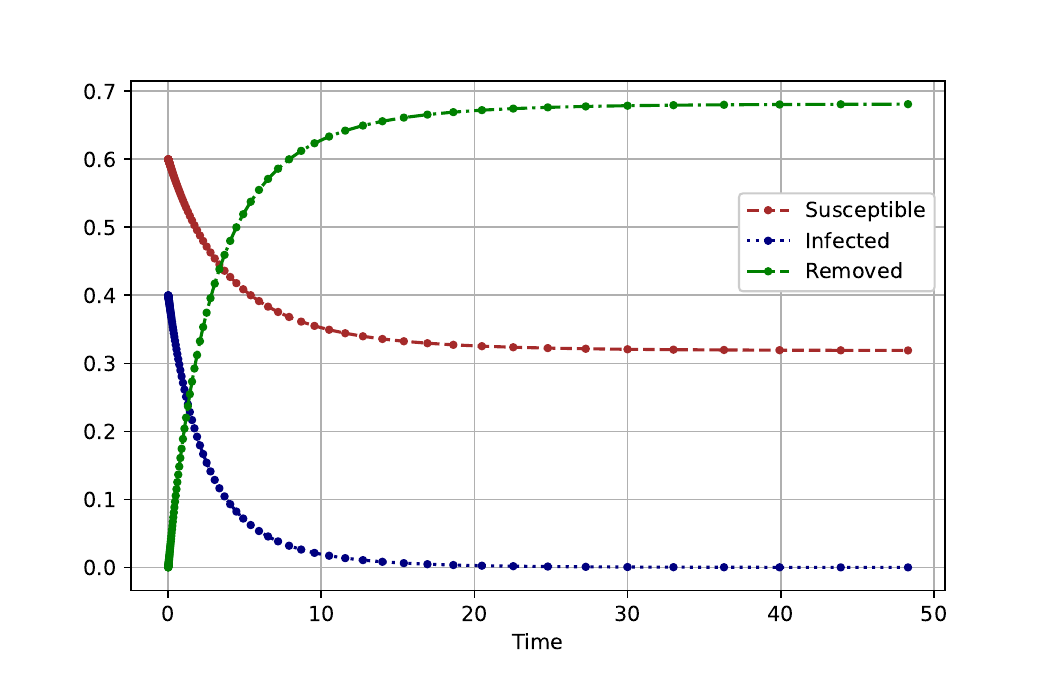}}
\caption{Susceptible, Infected and Removed individuals with $b = 0.3$, $c = 0.6$, 
$x(t_0) = 0.6$, $y(t_0) = 0.4$, $z(t_0) = 0$, $q = 1.1$, and $t_0 = 0.01$.}
\label{Fig.2}
\end{figure}

As expected, for $q \approx 1$, the obtained quantum solutions
closely approximate the solutions of the classical continuous SIR model 
(see Figures~\ref{Fig.1} and \ref{Fig.2}). Note that the parameter values 
and initial conditions used in the example shown in Figure~\ref{Fig.1}  are consistent 
with condition \eqref{condition}.

For higher values of $q$, in the specific case where $R_0 < 1$, 
the susceptible population $x$ continues to converge to a certain $\alpha \in \, ]0,N]$, 
although to an increasingly smaller positive value (see Figure~\ref{Fig.3}). 
This last result is in line with Lemma~\ref{lemma} and Theorem~\ref{thm:5}.  

\begin{figure}[ht!]
\centering
\includegraphics[scale=0.4]{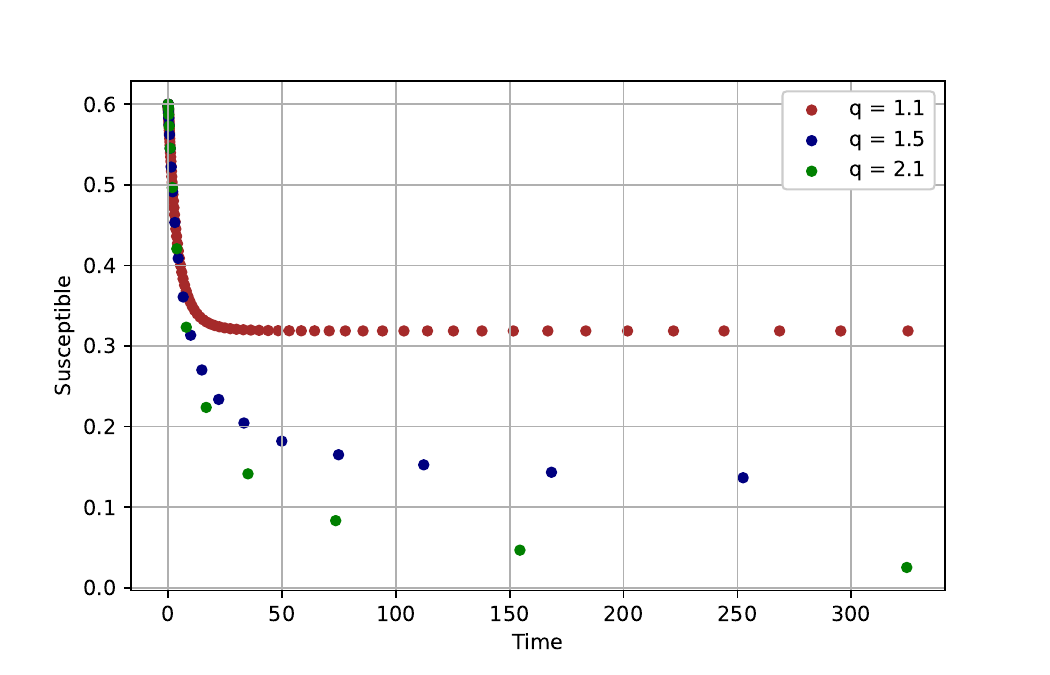}
\caption{Susceptible individuals for different values of $q$ and $b<c$.}
\label{Fig.3}
\end{figure}


\section{Conclusion}
\label{sec4}

In this work, we proposed a quantum SIR model considering the concept of $q$-derivative 
and a particular rule of Mickens' method. The equilibrium points were found, and 
the non-negativity and boundedness of solutions were proved. We derived 
the exact solution of the quantum SIR model and proved asymptotic stability. 
We end providing some examples that illustrate the exact solution obtained. 

While the assumption of constant rates is often not realistic in practice, 
having an exact solution for such a case can serve as a benchmark for comparing 
the behavior of the system under different parameter values. Our examples show that, 
for a $q \rightarrow 1^+$, our system behaves similarly to the continuous model. 
Moreover, our results highlight the potential of quantum time as an 
alternative framework for discrete modeling in epidemiology. Future work may include 
extending this approach to more complex compartmental models, as well as exploring its
applications in other scientific domains.


\section*{Declarations}

\subsection*{Conflict of interest}

The authors declare no conflicts of interest.


\subsection*{Data availability}

The manuscript has no associated data.


\subsection*{Funding} 

The authors were partially supported by 
the Portuguese Foundation for Science and Technology (FCT):
Lemos-Silva and Torres through the Center for Research and Development 
in Mathematics and Applications (CIDMA) of University of Aveiro (UA), 
project UID/04106/2025 (\url{https://doi.org/10.54499/UID/04106/2025}); 
Vaz through the Center of Mathematics and Applications (CMA) 
of Universidade da Beira Interior (UBI), 
project UID/00212/2025 (\url{https://doi.org/10.54499/UID/00212/2025}).
Lemos-Silva is also supported by the FCT PhD fellowship UI/BD/154853/2023
(\url{https://doi.org/10.54499/UI/BD/154853/2023}).


\subsection*{Acknowledgements}

The authors are grateful to two reviewers for reading
the initial submitted version carefully and for several 
comments and suggestions.



\end{document}